\documentclass[a4paper,leqno,12pt,twoside]{amsart}
\usepackage{amsmath,amsfonts,amssymb,amsthm,amscd}
\usepackage[utf8]{inputenc}
\usepackage[T1]{fontenc}
\usepackage[english]{babel}
\usepackage[colorlinks=true,citecolor=blue, urlcolor=blue, linkcolor=blue,pagebackref]{hyperref}
\usepackage[top=1.2in,bottom=1.2in,left=1.in,right=1.in]{geometry} 
\usepackage{graphicx}
\usepackage{paralist}
\usepackage{tabto}
\usepackage{standalone}
\usepackage{tikz}
\usetikzlibrary{matrix}
\usetikzlibrary{arrows}
\usepackage{xfrac}
\usepackage{amsthm, amsmath, amssymb,latexsym}
\usepackage{tikz-cd}
\usepackage[all]{xy} 

\usepackage{enumerate}
\usepackage{xcolor}
\usepackage{aliascnt}
\usepackage{cleveref}

\newtheorem{theorem}{Theorem}[section]
\newtheorem{lemma}[theorem]{Lemma}
\newtheorem{corollary}[theorem]{Corollary}
\newtheorem{proposition}[theorem]{Proposition}
\newtheorem{remark}[theorem]{Remark}
\newtheorem{definition}[theorem]{Definition}

\def\ra{\rightarrow}
\title{Gaussian maps on trigonal curves}

\begin{document}
 \author{A. Lacopo}
  \email{antonio.lacopo01@universitadipavia.it}
\begin{abstract}
    In this paper we study higher even Gaussian maps of the canonical bundle for cyclic trigonal curves. More precisely, we study suitable restrictions of these maps determining a lower bound for the rank, and more generally, a lower bound for the rank for the general trigonal curve. We also manage to give the explicit description of the kernel of $\mu_2$. Finally, we use these results to show the non existence of "extra" asymptotic directions for cyclic trigonal curves in some subspaces of $H^1(T_C)$ generated by higher Schiffer variations.
\end{abstract}

\thanks
{A. Lacopo is a member of GNSAGA (INdAM) and is partially supported by PRIN project {\em Moduli spaces and special varieties} (2022).}
\maketitle
\section{Introduction}
In this paper we will study higher Gaussian maps of the canonical bundle for some special trigonal curves, that are cyclic covers of $\mathbb{P}^1$. We will compute the rank of some suitable restrictions of even order maps and then we will use these results in order to show the non existence of "extra" asymptotic directions in some subspaces of $H^1(T_C)$ generated by higher Schiffer variations. \\ 
Let $C$ be a smooth projective curve and let us denote by $I_2(K_C)$ the kernel of the multiplication map 
$$\mu_0:Sym^2H^0(K_C)\rightarrow H^0(K_C^{\otimes 2}),$$
the second Gaussian map $\mu_2$ is a linear map $$\mu_2:I_2(K_C)\rightarrow H^0(K_C^{\otimes 4})$$ and more generally, for any $k\geq 2$, the $2k$-th Gaussian map is a linear map 
$$\mu_{2k}:Ker(\mu_{2k-2})\rightarrow H^0(K_C^{\otimes 2k+2}).$$ These maps were introduced by Wahl in \cite{Wahl1987TheJA}.\\ 
There are two main reasons why Gaussian maps are studied: the first one is the strong connection between even order maps and the second fundamental form of the Torelli map, and the second one is that surjectivity of not of these maps gives information about the geometry of the given variety. Let us now better explain these two reasons, beginning from the first one. \\  
Assume $g\geq 4$. Recall that we have the Torelli map
    $$
    j:\mathcal{M}_g\rightarrow \mathcal{A}_g
    $$
    $$
    [C]\mapsto [J(C),\Theta_C],
    $$
    mapping the class of a curve to its Jacobian, together with the polarization given by the Theta divisor. The map $j$ is injective (Torelli Theorem) and an orbifold embedding outside the hyperelliptic locus $\mathcal{H}_g$. We denote by $\mathcal{M}^0_g$ the complement of the hyperelliptic locus.\\
    Consider the tangent bundle exact sequence 
\begin{equation}
\label{tangent}
0 \rightarrow T_{{\mathcal M}^0_g} \stackrel{dj}\rightarrow  T_{{{\mathcal A}_{g}}_{|{\mathcal M}^0_g}} \stackrel{\pi}\rightarrow {\mathcal N_{{\mathcal M}^0_g/{\mathcal A}_{g}}} \rightarrow 0.
\end{equation}

Denote by  $$\sigma: T_{{\mathcal M}^0_g} \otimes T_{{\mathcal M}^0_g} \rightarrow {\mathcal N_{{\mathcal M}^0_g/{\mathcal A}_{g}}}, \ X \otimes Y \mapsto \pi (\nabla_X(Y))$$ its  second fundamental form. Denote by $\rho$ the dual of  $\sigma$,
$$\rho: {\mathcal N^{\vee}_{{\mathcal M}^0_g/{\mathcal A}_{g}}} \rightarrow Sym^2(\Omega^1_{{\mathcal M}^0_g}) .$$
At the point  $[C] \in {\mathcal M}_g^0$ the dual of the exact sequence  \eqref{tangent} is
$$0 \rightarrow I_2(K_C)  \rightarrow S^2H^0(K_C) \stackrel {m}\rightarrow H^0(K_C^{\otimes 2})
\rightarrow 0,$$ where $m$ is the multiplication map.  Hence at the point  $[C]$ we have a linear map
$$\rho:I_2(K_C)\rightarrow Sym^2(H^0(K_C^{\otimes 2})).$$  A result due to Colombo,Pirola,Tortora gives a strong relation between the second fundamental form $\rho$ and $\mu_2$:
    $$
    \rho(Q)(\xi_p\odot \xi_p)=-2\pi i\mu_2(Q)(p),
    $$
    where $p\in C$, $Q\in 
    I_2(K_C)$ and $\xi_p$ is the Schiffer variation at $p$.\\
    This result was later generalized by Frediani (see \cite{Frediani2025} for more details), showing a connection between $\rho$ and even order higher Gaussian maps. We recall the following definition, which is of great importance for our purpose.
	\begin{definition}
		A nonzero direction $\zeta \in H^1(T_C)$ is said to be asymptotic if $\rho(Q)(\zeta \odot \zeta) = 0$, $\forall Q \in I_2$. 
	\end{definition}
	  Remarkable results in the study of asymptotic directions in the trigonal case are due to Colombo-Frediani-Pirola, who classified the rank one case. They proved the following result.
	\begin{theorem}[\cite{ColomboFredianiPirola+2025}]
		If $C$ is a trigonal (non hyperelliptic) curve of genus $g\geq 8$, or of genus $g=6,7$ and Maroni degree $2$, the rank one asymptotic directions are exactly the Schiffer variations in the ramification points of the $g^1_3$.
	\end{theorem}
	Moreover, the authors also proved that for trigonal curves of genus $5$ or genus $g=6,7$ and Maroni degree $1$, there can exist rank $2$ asymptotic directions that are not Schiffer variations in the ramification points of the $g^1_3$.\\ \\
    As for the second reason, we recall some famous results. The first map which was studied was $\mu_1$, which is a map $$\mu_1:\Lambda^2H^0(K_C)\rightarrow H^0(K_C^{\otimes 4}).$$ 
    Wahl proved that if $C$ is a curve lying on a $K3$ surface, then $\mu_1$ is not surjective (see \cite{Wahl1987TheJA}). On the other hand, Ciliberto, Harris and Miranda proved that for the general curve of genus $g\geq 10, g\neq 11$, $\mu_1$ is indeed surjective. Similar results hold for $\mu_2:$ Colombo, Frediani and Pareschi in \cite{Colombo2009HyperplaneSO} proved that if $C$ lies on an abelian surface, $\mu_2$ is not surjective; while Calabri, Ciliberto and Miranda (\cite{10.1307/mmj/1320763048} proved the surjectivity of $\mu_2$ for the general curve of genus $g\geq 18$. Hence it looks like that surjectivity of Gaussian maps gives obstruction to the curve for being embedded in particular class of surfaces. \\
    For higher Gaussian maps $(\mu_k$, with $k\geq 3)$ less is known: the main result is due to Rios-Ortiz, who proved the surjectivity of $\mu_k$ for the general curve of genus high enough (see \cite{10.1093/imrn/rnad165}). \\ \\
	Our first main result is the computation of the rank of suitable restrictions of even order Gaussian maps: in order to do that we will follow the ideas of Section \cite{faro2025highergaussianmapshyperelliptic}. This computation allows us to give a lower bound for the rank of $\mu_{2k}$ on cyclic trigonal curves and, more generally, a lower bound for the rank of $\mu_{2k}$ on the general trigonal curve. 
	First, we recall the following known result, due to Colombo and Frediani (\cite{10.1307/mmj/1260475698}).
		\begin{theorem}
		Let $C$ be any trigonal curve of genus $g\geq 8$. Then $$rank(\mu_2)=4g-18.$$
	\end{theorem}
	From Theorem above, we see that we need to study the rank of $\mu_{2k}$ for $k\geq 2$. The result is the following.
	 
	\begin{theorem}(See Theorem \ref{theogentrig})
		\label{unotheo}
		Let $C$ be a general trigonal curve of genus $g\geq 16$, let $2\leq k\leq \lfloor\frac{g-4}{6}\rfloor$ be an integer. Then 
		$$rank(\mu_{2k})\geq 2g-8k-2.$$
	\end{theorem}
    Theorem \ref{unotheo} allows us to prove the following result.
    \begin{theorem}(See Theorem \ref{asytrig}
        Let $C$ be any cyclic trigonal curve of genus $g\geq 16$ and let $V\subset H^1(T_C)$ be the set of asymptotic directions. Then for a general point $p\in C$,  $V\cap \langle \xi_p^1,...,\xi_p^{\lfloor\frac{n_2}{2}\rfloor} \rangle=0$.
    \end{theorem}
    As a corollary, we obtained a bound for the dimension of a totally geodesic subvariety of $\mathcal{A}_g$ generically contained in the Torelli locus and passing through a cyclic trigonal curve. More precisely, we have the following result.
    \begin{theorem}(See Corollary \ref{dimensasy})
        Let $Y$ be a germ of a totally geodesic subvariety of $\mathcal{A}_g$ generically contained in $j(\mathcal{M}^0_g)$ and passing through $j(C)$, where $C$ is any cyclic trigonal curve of genus $g\geq 16$. Then we have
        $$dim(Y)=dim(W)\leq 3g-3-\lfloor\frac{n_2}{2}\rfloor.$$
    \end{theorem}
	Then we focus on the study of asymptotic directions: in order to do that, we will consider Schiffer variations on the ramification points of the $g^1_3$. We will restrict our attention to cyclic trigonal curves. We manage to prove the following theorem, which is the other main result of this paper.
	\begin{theorem}(See Theorem \ref{theoxi1x4})
		Let $C$ be any trigonal curve of genus $g\geq 10$ which is a cyclic cover of $\mathbb{P}^1$. For every ramification point $p$ of $\pi$, the only asymptotic directions in $\langle \xi_p^1,\xi_p^2,\xi_p^3,\xi_p^4\rangle$ are in $\langle\xi_p\rangle$.
	\end{theorem}
      Throughout this paper, in particular for the study of asymptotic directions, we will use many times a known result (\cite{Frediani2025}), which we state here for the reader's convenience. \\
      Let $\omega_1,...,\omega_g$ be a basis of $H^0(K_C)$, where locally $\omega_i=f_i(z)dz$. Let $Q\in I_2(K_C)$: then $Q=\sum\limits_{i,j=1}^ga_{ij}(\omega_i\otimes \omega_j)$ and $\sum\limits_{i,j=1}^ga_{ij}f_i(z)f_j(z)\equiv 0$.\\
    The result is the following (Remark 3.2 in \cite{Frediani2025}).
\begin{proposition}
\label{remarkfrediani}
    Let $Q\in I_2(K_C)$ with a local expression as above, such that 
    $$ \sum\limits_{i,j=1}^ga_{ij}f_i^{(h)}(0)f_j^{(k)}(0)= 0\quad \forall h+k\leq m,$$ then 
    $$ \rho(Q)(\xi_p^n\odot \xi_p^l)=0\quad if\quad n+l\leq m$$ and if $l+n=m+1$ we have 
    $$\rho(Q)(\xi_p^n \odot \xi_p^l)  =  2 \pi i  \left( \sum_{k=0}^{n-1} \left (  \sum^u_{i, j=1}a_{ij} f^{(m+1-k)}_i(0) f_j^{(k)}(0)\right)\frac{(n-k)}{k!(m+1-k)!}\right).$$
\end{proposition}
     Moreover, we managed to find the explicit descriptions of $Ker(\mu_2)$ for a particular cyclic trigonal curve (see Theorem \ref{mu2tensor}).
	\section{Gaussian maps on cyclic trigonal curves}
	As we said, here we will consider some particular class of trigonal curves, namely the ones that are cyclic coverings of $\mathbb{P}^1$. These curves have affine equation 
	\begin{equation}
		y^3=\prod_{i=1}^r (x-t_i)^{a_i}
	\end{equation}
	where $a_i=1,2$, $\sum_{i=1}^r a_i\equiv 0$(mod $3$) and the $t_i$'s are distinct roots. We will assume $t_1=0$. Clearly, for these curves we have a $3:1$ map $\pi:C\rightarrow \mathbb{P}^1$ where $\mathbb{P}^1=C/\mathbb{Z}/3\mathbb{Z}$.\\
	Recall that every ramification point of $\pi$ is total, hence by Hurwitz formula we get $g(C)=r-2$. We will consider every possible kind of families, which have the following equation:
    \begin{equation}
    \label{cyctrigonal}
        y^3=\prod_{i=1}^{r_1}(x-t_i)\prod_{i=r_1+1}^{r_1+r_2}(x-t_i)^2
    \end{equation}
    where $r_1+r_2=r$. An interesting particular case is when $r_1=r$ and $r_2=0$, and we have 
	\begin{equation}
		\label{trigonal1}
		y^3=\prod_{i=1}^r (x-t_i).
	\end{equation}
	Let us denote by $L$ the $g_3^1$ and $M=K_C\bigotimes L^{\vee}$. Moreover, let $H^0(L)=\langle s,t \rangle$ and $H^0(M)=\langle \omega_1,...,\omega_{g-2} \rangle$. For the moment, we will focus on the family of equation \eqref{trigonal1}. We know that $H^0(M)=W_1\bigoplus W_2$, where $$W_1=\langle \alpha_1=x\frac{dx}{y},\alpha_2=x^2\frac{dx}{y},...,\alpha_{n_1}=x^{n_1}\frac{dx}{y}\rangle$$ and 
	$$
	W_2=\langle \beta_1=x\frac{dx}{y^2},\beta_2=x^2\frac{dx}{y^2},...,\beta_{n_2}=x^{n_2}\frac{dx}{y^2}\rangle
	$$
	where $n_1=\frac{r-6}{3}, n_2=\frac{2r-6}{3}$ (see for instance \cite{moonen}). \\
	In this section we will compute the equations of the kernel of even order higher gaussian maps  and their rank. More precisely, let $k\geq 0$ be an integer and let us consider the even order higher Gaussian map of the canonical bundle $\mu_{2k}: Ker (\mu_{2k-2})\rightarrow H^0(K_C^{\bigotimes 2k+2})$. Following the ideas of \cite{faro2025highergaussianmapshyperelliptic} we can compute the rank of a suitable restriction of $\mu_{2k}$. Moreover, we will manage to give a complete description of the equations of $Ker(\mu_2)$: we mention that the computation of the rank of $\mu_2$ was already known by a result of Colombo-Frediani (\cite{10.1307/mmj/1260475698}).\\
	We begin with some remarks.
	\begin{remark}
		From  (\cite{faro2025highergaussianmapshyperelliptic}, Lemma 3.5), it follows that $rk(\mu_{2k})=rk(\mu_{2k-1,M})$, because $\mu_{1,L}(s \wedge t)$ is a nonzero section in $H^0(K_C\otimes L^2)$.
	\end{remark}
	\begin{remark}
		Let $\mathbb{Z}/3\mathbb{Z}=\langle \zeta\rangle.$
		Since $\mu_{1,M}:\Lambda^2W_1\bigoplus \Lambda^2W_2\bigoplus W_1\otimes W_2\rightarrow H^0(K_C^3\otimes L^{-2})$ is equivariant (with respect to the action of $\mathbb{Z}/3\mathbb{Z}$), we have the following:
		\begin{equation*}
			\mu_{1,M|\Lambda^2 W_1}:\Lambda^2 W_1\rightarrow H^0(K_C^3\otimes L^{-2})^2
		\end{equation*}
		\begin{equation*}
			\mu_{1,M|\Lambda^2 W_2}:\Lambda^2 W_2\rightarrow H^0(K_C^3\otimes L^{-2})^1
		\end{equation*}
		\begin{equation*}
			\mu_{1,M|W_1\otimes W_2}:W_1 \otimes W_2\rightarrow H^0(K_C^3\otimes L^{-2})^0
		\end{equation*}
		where by $H^0(K_C\otimes L^2)^i$ we mean the invariant part with respect to the multiplication by $\zeta^i$. Clearly, by restriction the same splitting holds for $\mu_{2k-1}$ for every $k\geq 0$.
	\end{remark}
	Our goal is now to compute $rk(\mu_{2k-1\mid \Lambda^2 W_i})$. First, we recall the following classical result (see for instance \cite{Andreotti1967}).
    \begin{lemma}
        \label{lemmaiso}
        Let $C$ be a trigonal curve of genus $g \geq 4$. Let $|L|$ be the $g^1_3$. Set $M=K_C \otimes L^{\vee}$, let $\omega_1,...,\omega_{g-2}$ be a basis for $H^0(M)$ and let $\langle s,t \rangle$ be a basis for $H^0(L)$.  Then the map defined by 
\begin{align}
\Lambda^2H^0(M) &\xrightarrow{\psi} I_2 \nonumber
\\ \nonumber 
\omega_i \wedge \omega_j &\ra Q_{ij}:=s\omega_i \odot t\omega_j - s\omega_j \odot t\omega_i,\nonumber 
\end{align} 
is an isomorphism. In particular, observe that $\{Q_{ij}\}$ gives a basis for $I_2$.
    \end{lemma}
	\begin{remark}
		\label{qij}
		Let $Q\in I_2$. From Lemma \ref{lemmaiso}, recalling that $\Lambda^2H^0(M)=\Lambda^2W_1\bigoplus \Lambda^2W_2\bigoplus W_1\otimes W_2$, we have $$Q=\sum\limits_{1\leq i<j\leq n_1} a_{ij}Q_{ij}+\sum\limits_{1\leq i<j\leq n_2} a'_{ij}Q'_{ij}+\sum \limits_{\substack{1 \leq i\leq n_1  \\ 1\leq j \leq n_2}} a''_{ij}Q''_{ij},$$ where $$Q_{ij}:=s\alpha_i \odot t\alpha_j - s\alpha_j \odot t\alpha_i,$$ $$Q'_{ij}:=s\beta_i \odot t\beta_j - s\beta_j \odot t\beta_i,$$ $$Q''_{ij}:=s\alpha_i \odot t\beta_j - t\alpha_j \odot s\beta_i.$$
	\end{remark}
    The first result is the following.
	\begin{proposition}
		\label{eqmu2k-1}
		Let $C$ be a trigonal curve of genus $g\geq 4$ as in \eqref{trigonal1} and let $\psi$ be the isomorphism in Lemma \ref{lemmaiso}.\\ Let $Q=\sum\limits_{1\leq i<j\leq n_1} a_{ij}Q_{ij}\in I_2$. For any $k \geq 2$,  $\psi ^{-1}(Q) \in Ker(\mu_{2k-1\mid \Lambda^2 W_1})$ if and only if 
		\begin{itemize}
			\item $\forall \  3\leq l \leq 2n_1-1$,
			$$
			\sum \limits_{\substack{1 \leq i < j \\ i+j=l}} a_{ij} (j-i)=0,
			$$
			\item  $\forall \  2 \leq m \leq k$, $\forall \  2m-1 \leq l \leq 2n_1-1$,
			\begin{equation}
				\label{equationsmuk-1}
				\sum \limits_{\substack{1 \leq i < j \\ i+j=l \\ i \geq m-1, \\ j \geq m-1 }} a_{ij} (j-i) ij(i-1)(j-1)...(i-(m-2))(j-(m-2))=0.
			\end{equation}
		\end{itemize}
		Analogously, let $Q=\sum\limits_{1\leq i<j\leq n_2} a'_{ij}Q'_{ij}\in I_2$. Then $\psi ^{-1}(Q) \in Ker(\mu_{2k-1\mid \Lambda^2 W_2})$ if and only if 
		\begin{itemize}
			\item $\forall \  3\leq l \leq 2n_2-1$,
			$$
			\sum \limits_{\substack{1 \leq i < j \\ i+j=l}} a'_{ij} (j-i)=0,
			$$
			\item  $\forall \  2 \leq m \leq k$, $\forall \  2m-1 \leq l \leq 2n_2-1$,
			\begin{equation}
				\sum \limits_{\substack{1 \leq i < j \\ i+j=l \\ i \geq m-1, \\ j \geq m-1 }} a'_{ij} (j-i) ij(i-1)(j-1)...(i-(m-2))(j-(m-2))=0.
			\end{equation}
		\end{itemize}
	\end{proposition}
	\begin{proof}
		The proof goes exactly as in (\cite{faro2025highergaussianmapshyperelliptic}, Lemma 3.7). In fact, since $\Lambda^2H^0(M)=\Lambda^2W_1\bigoplus \Lambda^2W_2\bigoplus W_1\otimes W_2$ and we are only considering the restrictions of $\mu_{2k-1}$ to $\Lambda^2 W_i$, the role of $H^0(M)$ in (\cite{faro2025highergaussianmapshyperelliptic}, Lemma 3.7) is taken by $W_i$ for $i=1,2$ in our context. The computations are exactly the same. 
	\end{proof}
    Hence we obtain the analogous of (\cite{faro2025highergaussianmapshyperelliptic}, Theorem 3.1) for these trigonal curves.
    \begin{theorem}
    \label{rango12}
        Let $C$ be a cyclic trigonal curve of genus $g\geq 16$ as in \eqref{trigonal1}. Then for every $2\leq k\leq \lfloor\frac{n_i}{2}\rfloor$
		\begin{equation}
			rk(\mu_{2k-1\mid \Lambda^2 W_i})=2n_i-4k+1
		\end{equation}
		\begin{equation}
			dim(Ker(\mu_{2k-1\mid \Lambda^2 W_i}))=\frac{n_i(n_i-1)}{2}-k(2n_i-2k-1)
		\end{equation}
		Moreover, if $k>\lfloor\frac{n_i}{2}\rfloor$, then $rk(\mu_{2k-1\mid \Lambda^2 W_i})=0$.
	\end{theorem}
    Notice that on Theorem \ref{rango12} we need to assume $g\geq 16$ in order to have $k\geq 2$. In fact $\lfloor\frac{n_1}{2}\rfloor=\lfloor\frac{2r_1+r_2-6}{6}\rfloor=\lfloor\frac{r+r_1-6}{6}\rfloor=\lfloor\frac{g-4}{6}+\frac{r_1}{6}\rfloor\geq \lfloor\frac{g-4}{6}\rfloor\geq k$. So we get $\lfloor\frac{g-4}{6}\rfloor\geq 2$, which is true only if $g\geq 16$.\\
    In what follows for simplicity we sometimes write $Q\in Ker(\mu_{2k+1})$ for some $k$: in this case we always mean $\psi^{-1}(Q)\in Ker(\mu_{2k+1})$.
	\begin{remark}
		\label{trig2trig3}
		Before going on, we want to highlight that Proposition \ref{eqmu2k-1} holds for every cyclic trigonal curve, not only for the ones as in \eqref{trigonal1}.
        \begin{proof}
        Following \cite{moonen}, for a cyclic trigonal curve of equation \eqref{cyctrigonal}, we have
		  $H^0(M)=W_1\bigoplus W_2$, where $$W_1=\langle x\frac{dx}{y},x^2\frac{dx}{y},...,x^{n_1}\frac{dx}{y}\rangle$$ and 
			$$
			W_2=\langle x\prod_{i=r_1+1}^{r}(x-t_i)\frac{dx}{y^2},x^2\prod_{i=r_1+1}^{r}(x-t_i)\frac{dx}{y^2},...,x^{n_2}\prod_{i=r_1+1}^{r}(x-t_i)\frac{dx}{y^2}\rangle
			$$
			where $n_1=\frac{r_1+2r_2-6}{3},n_2=\frac{2r_1+r_2-6}{3}$. \\
		First notice that $W_1$ is always the same, hence the Proposition clearly holds for every family if  $Q\in Ker(\mu_{2k-1\mid \Lambda^2 W_1})$. \\ The case of $W_2$ is  slightly less obvious: notice that here the holomorphic forms are $\prod_{i=r_1+1}^{r}(x-t_i)\beta_i$. If we call $g:=\prod_{i=r_1+1}^{r}(x-t_i)$ the common factor of the forms, a direct computation gives
        \begin{equation}
            \label{conto}
            \mu_{2k-1}(\sum\limits a'_{ij}(g\beta_i\wedge g\beta_j))=g^2\mu_{2k-1}(\sum\limits a'_{ij}(\beta_i\wedge \beta_j)).
        \end{equation}
        We now prove \eqref{conto}, which immediately implies the claim. First, we have $$(g\beta_i)^{(r)}=\sum\limits_{l=0}^r\binom{r}{l}g^{(r-l)}\beta_i^{(l)}.$$ Assume $r+s=2k-1$: we prove the claim by induction on $k$. The case $k=1$ is easy: indeed, we have $$\mu_1((g\beta_i)\wedge (g\beta_j))=(g\beta_i)'(g\beta_j)-(g\beta_j)'(g\beta_i)=(g'\beta_i+g\beta_i')(g\beta_j)-(g'\beta_j+g\beta_j')=g^2(\beta_i'\beta_j-\beta_j'\beta_i).$$ For $k\geq 2$ we get
		\begin{multline*}
			\mu_{2k-1}(\sum\limits a'_{ij}(g\beta_i\wedge g\beta_j))=\sum_{i,j}a'_{ij}((g\beta_i)^{(r)}(g\beta_j)^{(s)}-(g\beta_i)^{(s)}(g\beta_j)^{(r)})=\\
			=\sum\limits_{i,j}a'_{ij}(\sum\limits_{l=0}^r\binom{r}{l}g^{(r-l)}\beta_i^{(l)}\sum\limits_{m=0}^s\binom{s}{m}g^{(s-m)}\beta_i^{(m)}-\sum\limits_{m=0}^s\binom{s}{m}g^{(s-m)}\beta_i^{(m)}\sum\limits_{l=0}^r\binom{r}{l}g^{(r-l)}\beta_i^{(l)})=\\=\sum\limits_{l=0}^r\sum\limits_{m=0}^s\binom{r}{l}\binom{s}{m}g^{(r-l)}g^{(s-m)}\sum\limits_{i,j}a'_{ij}(\beta_i^{(l)}\beta_j^{(m)}-\beta_i^{(m)}\beta_j^{(l)}).
		\end{multline*}
		Let us focus on the last term. By the inductive hypothesis, we get $\mu_{2k-3}(\sum\limits a'_{ij}(g\beta_i\wedge g\beta_j))=g^2\mu_{2k-3}(\sum\limits a'_{ij}(\beta_i\wedge \beta_j))$, hence $\sum\limits a'_{ij}(\beta_i\wedge \beta_j)\in Ker(\mu_{2k-3})$.\\ So we have $\sum\limits_{i,j}a'_{ij}(\beta_i^{(l)}\beta_j^{(m)}-\beta_i^{(m)}\beta_j^{(l)})=0$ for every $l,m$ such that $l+m<r+s=2k-1$. Hence the only remaining terms are those where $l+m=r+s$, which gives $l=r,m=s$. Finally, we get
		$$\mu_{2k-1}(\sum\limits a'_{ij}(g\beta_i\wedge g\beta_j))=g^2\sum\limits_{i,j}a'_{ij}(\beta_i^{(r)}\beta_j^{(s)}-\beta_i^{(s)}\beta_j^{(r)})=g^2\mu_{2k-1}(\sum\limits a'_{ij}(\beta_i\wedge \beta_j)),$$
		as desired. So the claim follows because $g\not\equiv 0$, hence the equations of $Ker(\mu_{2k})$ are the same for every family.
        \end{proof}
	\end{remark}
	Joining Theorem \ref{rango12} and Remark \ref{trig2trig3}, we have the following result.
	\begin{theorem}
		\label{rank}
		Let $C$ be any cyclic trigonal curve of genus $g\geq16$. Then for every $2\leq k\leq \lfloor\frac{n_i}{2}\rfloor$
		\begin{equation}
			rk(\mu_{2k-1\mid \Lambda^2 W_i})=2n_i-4k+1
		\end{equation}
		\begin{equation}
			dim(Ker(\mu_{2k-1\mid \Lambda^2 W_i}))=\frac{n_i(n_i-1)}{2}-k(2n_i-2k-1)
		\end{equation}
		Moreover, if $k>\lfloor\frac{n_i}{2}\rfloor$, then $rk(\mu_{2k-1\mid \Lambda^2 W_i})=0$.
	\end{theorem}
	From Theorem above we get different corollaries.\\
	First, we immediately obtain the following result.
	\begin{corollary}
		Let $C$ be any cyclic trigonal curve of genus $g\geq 16$, let $2\leq k\leq \lfloor\frac{n_1}{2}\rfloor$ be an integer. Then 
		$$rank(\mu_{2k})\geq 2g-8k-2.$$
		While if $\lfloor\frac{n_1}{2}\rfloor < k\leq \lfloor\frac{n_2}{2}\rfloor$, then 
		$$rank(\mu_{2k})\geq 2n_2-4k+1.$$
	\end{corollary}
	\begin{proof}
		In the first case, from Theorem \ref{rank} we get 
		$$rank(\mu_{2k})\geq 2n_1-4k+1+2n_2-4k+1=2g-8k-2.$$
		The second part is obvious by Theorem \ref{rank}. 
	\end{proof}
	More generally, Theorem \ref{rank} gives us a lower bound for the rank of $\mu_{2k}$ on the general trigonal curve.
	\begin{theorem}
		\label{theogentrig}
		Let $C$ be a general trigonal curve of genus $g\geq 16$, let $2\leq k\leq \lfloor\frac{g-4}{6}\rfloor $ be an integer. Then 
		$$rank(\mu_{2k})\geq 2g-8k-2.$$
	\end{theorem}
	Finally, Theorem \ref{rank} allows us to show the non existence of asymptotic directions in a certain subspace of $H^1(T_C)$.
	\begin{theorem}
		\label{asytrig}
		Let $C$ be any cyclic trigonal curve of genus $g\geq 16$ and let $V\subset H^1(T_C)$ be the set of asymptotic directions. Then for a general point $p\in C$,  $V\cap \langle \xi_p^1,...,\xi_p^{\lfloor\frac{n_2}{2}\rfloor} \rangle=0$.
	\end{theorem}
	\begin{proof}
		We will follow the idea of (\cite{Frediani2025},[Theorem 3.4]).\\
		From Theorem \ref{rank} we have $Ker(\mu_{2k-1})\subsetneq Ker(\mu_{2k-3})$ if $1\leq k\leq \lfloor\frac{n_2}{2}\rfloor$, hence for every $i=1,...,\lfloor\frac{n_2}{2}\rfloor-1$ we can choose $Q_i\in Ker(\mu_{2i-1})$ to be a quadric such that $Q_i\notin Ker(\mu_{2i+1})$. Notice that here we need $\lfloor\frac{n_2}{2}\rfloor\geq 2$, which is true since $g\geq 16$. In fact, $\lfloor\frac{n_2}{2}\rfloor=\lfloor\frac{2r_1+r_2-6}{6}\rfloor=\lfloor\frac{r+r_1-6}{6}\rfloor=\lfloor\frac{g-4}{6}+\frac{r_1}{6}\rfloor\geq \lfloor\frac{g-4}{6}\rfloor\geq 2$. Let $U:=C\setminus \cup_{i=1}^{\lfloor\frac{n_2}{2}\rfloor-1} Z(\mu_{2i}(Q_i))$, where by $Z(\mu_{2i}(Q_i))$ we mean the zero set of the section. Let $p\in U$ and let $v\neq 0\in V$. Let $v=a_1\xi_p^1+...+a_{\lfloor\frac{n_2}{2}\rfloor}\xi_p^{\lfloor\frac{n_2}{2}\rfloor}$ and let $1\leq k\leq \lfloor\frac{n_2}{2}\rfloor$ be the maximum such that $a_k\neq 0$. Then we have 
		\begin{equation}
			\rho(Q_{k-1})(v\odot v)=\sum_{i,j=1}^k a_ia_j\rho(Q_{k-1})(\xi_p^i\odot \xi_p^j)=a_k^2\rho(Q_{k-1})(\xi_p^k\odot \xi_p^k)=a_k^2c_k\mu_{2k-1}(Q_{k-1})(p)\neq 0
		\end{equation}
		by the choice of $Q_{k-1}$ and $p$. This is a contradiction, since $v$ is asymptotic.
	\end{proof}
	Theorem \ref{asytrig} also gives us a bound for the dimension of a totally geodesic subvariety of $\mathcal{A}_g$ generically contained in the Torelli locus and passing through a cyclic trigonal curve. More precisely, we have the following corollary.
	\begin{corollary}
		\label{dimensasy}
		Let $Y$ be a germ of a totally geodesic subvariety of $\mathcal{A}_g$ generically contained in $j(\mathcal{M}^0_g)$ and passing through $j(C)$, where $C$ is a trigonal curve of genus $g\geq 16$ as in \eqref{cyctrigonal}. Then we have
        $$dim(Y)=dim(W)\leq 3g-3-\lfloor\frac{n_2}{2}\rfloor.$$
        In particular, for curves as in \eqref{trigonal1} we have 
        $$dim(Y)\leq \lfloor \frac{8g-8}{3}\rfloor.$$
	\end{corollary}
	\begin{proof}
		Set $W:=T_{j(C)}Y\subset H^1(T_C)$ and let $V:=\langle\xi_p^1,...,\xi_p^{\lfloor\frac{n_2}{2}\rfloor}\rangle$. Since a tangent vector to a totally geodesic variety is asymptotic, from Theorem \ref{asytrig}
		we get $W\cap V=(0)$. So we get 
		$$dim(V)+dim(W)=dim(V+W)\leq dim(H^1(T_C))=3g-3,$$
		hence $$dim(Y)=dim(W)\leq 3g-3-\lfloor\frac{n_2}{2}\rfloor.$$ In order to obtain the smallest possible dimension, we need to take $n_2$ as big as possible. This is precisely the case for curves as in \eqref{trigonal1}. Recalling that $n_2=\frac{2r-6}{3}$ and $r=g+2$, we have 
		$$dim(Y)\leq 3g-3-\lfloor\frac{g-1}{3}\rfloor=\lfloor \frac{8g-8}{3}\rfloor.$$
	\end{proof}
	\begin{remark}
		Unfortunately, the bound that we obtained in Corollary \ref{dimensasy} is strictly greater than the one in (\cite{doi:10.1142/S0129167X15500056}, Theorem 4.3). The authors proved that if $Y$ is a totally geodesic subvariety of $\mathcal{A}_g$ generically contained in the Torelli locus and passing through a trigonal curve, then $dim(Y)\leq 2g-1$.
	\end{remark}
	Now our goal is to describe the equations of $\mu_2$ for curves as in \eqref{trigonal1}. By Theorem \ref{rank}, we only need to compute the equations of $\mu_{1|W_1\otimes W_2}$. Let us define $f:=\frac{1}{y}$. Then we have 
	$$\mu_{1,M}(\alpha_i\wedge\beta_j)=(fx^i)'(f^2x^j)-(fx^i)(f^2x^j)'=f^3x^{i+j-1}[(i-j)-\frac{f'}{f}],$$ where the last equality follows after some easy computations. Hence, if $$\psi^{-1}(Q)=\sum \limits_{\substack{1 \leq i\leq n_1  \\ 1\leq j \leq n_2}} a_{ij}(\alpha_i\wedge \beta_j)$$ we get 
	\begin{equation}
		\label{mu1}
		\mu_1(\psi^{-1}(Q))=f^3\sum \limits_{\substack{1 \leq i\leq n_1  \\ 1\leq j \leq n_2}}a_{ij}x^{i+j-1}[(i-j)-\frac{f'}{f}].
	\end{equation}
	In order to simplify notation, we denote $h(x):=\prod_{i=1}^r(x-t_i)$. By Equation \eqref{trigonal1}, notice that we have $\frac{f'}{f}=-\frac{h'}{3y^3}=-\frac{h'}{3h}.$ Substituting in Equation \eqref{mu1}, we obtain 
	\begin{equation}
		\mu_1(\psi^{-1}(Q))=\frac{f^3}{3h}\sum \limits_{\substack{1 \leq i\leq n_1  \\ 1\leq j \leq n_2}}a_{ij}x^{i+j-1}[3h(i-j)+xh'].
	\end{equation}
	Hence 
	\begin{equation}
		\label{poly}
		\mu_1(Q)=0 \Leftrightarrow \sum \limits_{\substack{1 \leq i\leq n_1  \\ 1\leq j \leq n_2}}a_{ij}x^{i+j-1}[3h(i-j)+xh']=0.
	\end{equation}
	The expression in Equation \eqref{poly} is a polynomial in $x$, which is zero if and only if every coefficient is zero. In order to understand what the coefficients are, we need to write explicitly $h(x)$ and $h'(x)$. We denote by $\sigma_t$ the elementary symmetric function in $t$ variables. We have 
	$$h(x)=\sum\limits_{t=0}^{r-1}\sigma_tx^{r-t},$$
	$$xh'(x)=\sum\limits_{t=0}^{r-1}(r-t)\sigma_{r-t}x^{r-t}.$$
	So, substituting in Equation \eqref{poly} we get the following expression: 
	\begin{equation}
		\label{poly1}
		\mu_1(Q)=0\Longleftrightarrow \sum\limits_{s=2}^{n_1+n_2+r-1}x^s\sum \limits_{\substack{0 \leq t\leq r-1  \\ 2\leq l \leq n+m\\ t+s+1=l+r}}\sigma_t\sum \limits_{\substack{1 \leq i\leq n_1 \\ 1\leq j \leq n_2\\ i+j=l}}a_{ij}(3(i-j)+r-t)=0.
	\end{equation}
	\begin{theorem}
		\label{mu2tensor}
		Let $C$ be a trigonal curve as in \eqref{trigonal1}. Then 
		$$ rank(\mu_{1|W_1\otimes W_2})=2r-12.$$
	\end{theorem}
	\begin{proof}
		From Equation \eqref{poly1}, finally we get that $Q\in Ker(\mu_2)$ if and only if for every $2\leq s\leq n_1+n_2+r-1$ we have
		\begin{equation}
			\label{poly2}
			\sum \limits_{\substack{0 \leq t\leq r-1  \\ 2\leq l \leq n_1+n_2\\ t+s+1=l+r}}\sum \limits_{\substack{1 \leq i\leq n_1 \\ 1\leq j \leq n_2\\ i+j=l}}a_{ij}(3(i-j)+r-t)=0.
		\end{equation}
		We need to compute the number of linearly independent equations in \eqref{poly2}. So let us fix $l=i+j$. Notice that if $l=2$ or $l=n_1+n_2=r-4$ we only get one equation, which is either $$a_{11}=0$$ or $$a_{n_1n_2}=0.$$ Let us assume now $3\leq l\leq r-5$. For any $l$, a priori we have several different equations which depend on the value of $s$: more precisely we have $r-t=s+1-l$ and $s$ can increase at each step inside its range. We claim that, for fixed $l$, only two of these equations are linearly independent. Indeed, they are of the following form:
		\begin{equation}
			\begin{cases}
				a_{i_1j_1}(3(i_1-j_1)+r-t_1)+...+a_{i_uj_u}(3(i_u-j_u)+r-t_1)=0\\
				a_{i_1j_1}(3(i_1-j_1)+r-t_1+1)+...+a_{i_ub_u}(3(i_u-j_u)+r-t_1+1)=0\\ \vdots \\
				a_{i_1b_1}(3(i_1-j_1)+r-t_1+\delta)+...+a_{i_uj_u}(3(i_u-j_u)+r-t_1+\delta)=0\\
			\end{cases}
		\end{equation}
		where $i_k+j_k=l$, $u$ is the number of such couples $(i_k,j_k)$ and $\delta>0$ is an integer. Now, subtracting the first one from the second one we get $$a_{i_1j_1}+...+a_{i_uj_u}=0.$$ Notice that the same happens by subtracting the first one from any other equation, hence the first two are the only ones which are linearly independent. 
		Hence we proved that the number of linearly independent equations is $2+2(r-5-2)=2r-12$, which ends the proof.
	\end{proof}
	Notice that joining Theorem \ref{rank} and Theorem \ref{mu2tensor}, we obtain a new proof of the following result, which was already known by Frediani and Colombo (\cite{10.1307/mmj/1260475698}).
	\begin{theorem}
		Let $C$ be a trigonal curve as in \eqref{trigonal1}. Then $$rank(\mu_2)=2n_1-3+2n_2-3+2r-12=4g-18.$$
	\end{theorem}
	%Now let $p\in C$ be a ramification point of $\pi$. (AGGIUNGI REFERENZA-Coppens) We know that its gap sequence is the following 
	%\begin{equation}
	%   G_p=\{1,2,4,5,...,3n-2,3n-1,3n+1,3n+4,...,3(g-n-1)+1\}.
	%\end{equation}
	%where $n=\frac{g-1}{3}$.
	%We use this sequence in order to construct a %particular basis of $H^0(M)$.
	%\begin{lemma}
	%\label{dimensions}
	%    We have 
	%  \begin{equation}
		%      \begin{cases}
			%     h^0(M(-(3k+s)p)=g-2(k+1)-s
			%        \qquad   0\leq k\leq n-2, s=0,1,2; \quad  k=n-1, s=0 \\
			%        h^0(M(-(g-4+3k)p)=n+1-k \quad 0\leq k \leq n-1 \\
			%       h^0(M(-(2g-7)p)=h^0(M(-(2g-6)p)=1
			% \end{cases}
		%\end{equation}
		%\end{lemma}
		\section{Asymptotic directions}
		\label{asymptotictrig}
		In this section we will deal with Schiffer variations on the ramification points of the $g^1_3$ of our special trigonal curves. We will use some results of the previous section to make some computations of the second fundamental form on these kind of (higher) Schiffer variations. We will show the non existence of extra asymptotic directions in some subspaces of $H^1(T_C)$ (see Theorem \ref{theochap3}). \\
		Throughout this section, let $p=(0,0)\in C$ be a ramification point of $\pi$, corresponding to $t_1=0$.
		\begin{remark}
			\label{basis}
			Assume now that $C$ is as in equation \eqref{trigonal1}. Recall that $W_1=\langle x\frac{dx}{y},x^2\frac{dx}{y},...,x^{n_1}\frac{dx}{y}\rangle $ and 
			$W_2=\langle x\frac{dx}{y^2},x^2\frac{dx}{y^2},...,x^{n_2}\frac{dx}{y^2}\rangle$ and we defined $\alpha_i=\frac{x^idx}{y}$ for $i=1,...,n_1$ and $\beta_j:=\frac{x^jdx}{y^2}$ for $j=1,...,n_2$.\\
			Note that we have $$ord_p(\alpha_i)=3i-2,$$ $$ord_p(\beta_j)=3j-3.$$
			Moreover, note that since $p$ is a Weierstrass point, we can choose a basis of $H^0(L)$ given by $s\in H^0(L(-3p))$ and $t\in H^0(L)\setminus H^0(L(-p))$.
		\end{remark}
		Let $Q\in I_2$. From Remark \ref{qij}, we have $$Q=\sum\limits_{1\leq i<j\leq n_1} a_{ij}Q_{ij}+\sum\limits_{1\leq i<j\leq n_2} a'_{ij}Q'_{ij}+\sum \limits_{\substack{1 \leq i\leq n_1  \\ 1\leq j \leq n_2}} a''_{ij}Q''_{ij} =\sum \limits_{r,m=1}^u c_{rm} \gamma_r \otimes \gamma_m,$$ where $\gamma_1,...,\gamma_u$ are elements of $H^0(K_C)$ and $$Q_{ij}:=s\alpha_i \odot t\alpha_j - s\alpha_j \odot t\alpha_i,$$ $$Q'_{ij}:=s\beta_i \odot t\beta_j - s\beta_j \odot t\beta_i,$$ $$Q''_{ij}:=s\alpha_i \odot t\beta_j - t\alpha_j \odot s\beta_i.$$
		\begin{proposition}
			\label{vanishing}
			Let $C$ be a trigonal curve of genus $g\geq 4$ as in \eqref{trigonal1}, $p\in C$ a Weierstrass point and let $Q\in I_2$. Let $\gamma_1,...,\gamma_u \in H^0(K_C)$ where locally $\gamma_r=g_rdz$ such that $$Q=\sum \limits_{r,m=1}^u c_{rm} \gamma_r \otimes \gamma_m=\sum \limits_{1\leq i<j\leq n_1} a_{ij}Q_{ij}+\sum \limits_{1\leq i<j\leq n_1} a'_{ij}Q'_{ij}+\sum \limits_{\substack{1 \leq i\leq n_2  \\ 1\leq j \leq n_2}} a''_{ij}Q''_{ij}$$ Then $$
			\sum \limits_{r,m=1}^u c_{rm} g_r^{(h)} g_m^{(l)}(p)=0 \quad \forall h+l\leq 3.
			$$
		\end{proposition}
		
		\begin{proof}
			We have 
			\begin{multline}
				\label{eq2}
				\sum \limits_{r,m=1}^u c_{rm} g_r^{(h)} g_m^{(l)}(p)=\sum\limits_{n=0}^h\sum\limits_{m=0}^l  \binom{h}{n} \binom{l}{m}(s^{(h-n)}t^{(l-m)}-t^{(h-n)}s^{(l-m)})(p)\cdot \\ [\sum \limits_{1\leq i<j\leq n_1} a_{ij}(\alpha_i^{(n)}\alpha_j^{(m)}-\alpha_j^{(n)}\alpha_i^{(m)})(p)+ \\ +\sum \limits_{1\leq i<j\leq n_2} a'_{ij}(\beta_i^{(n)}\beta_j^{(m)}-\beta_j^{(n)}\beta_i^{(m)})(p)+\sum \limits_{\substack{1 \leq i\leq n_1 \\   1\leq j \leq n_2}} a''_{ij}(\alpha_i^{(m)}\beta_j^{(n)}-\alpha_j^{(n)}\beta_i^{(m)})(p)].
			\end{multline}
			where $h+l\leq 3$. We want to show that the right hand of equation \eqref{eq2} is zero. Actually, we only need to check the case $h+l=3$. In fact, if $h+l=0$, this is just because $Q\in I_2$. The case $h+l=1$ is zero, since $\mu_1$ vanishes on symmetric tensors; while for $h+l=2$ we get $0$ because $p$ is a Weierstrass point and hence $(s't-t's)(p)=0$. So let us assume $h+l=3$. Since $n\leq h$ and $m\leq l$, we have $n+m\leq 3$.
			
			\begin{itemize}
				\item The case $n+m=3$ is trivial, since then $n=h$ and $m=l$, hence we get (up to scalar) $(st-ts)=0$.
				\item If $n+m=2$ we either have $n=h-1, m=l$ or the symmetric condition, interchanging the roles of $n,m$. In this case, we get $(s't-t's)(p)=0$ because of the choice of $s$.
				\item If $n+m=1$, then $n=h-2, m=l$ or the reverse condition, or $n=h-1,m=l-1$. In the first case we get $(s''t-t''s)(p)=0$, again because $s\in H^0(L(-3p))$, while in the second one we have $(s't'-t's')=0$.
				\item Finally, if $n=m=0$, then every summand inside the square brackets is zero. In fact, in the first two summands we get respectively $(\alpha_i\alpha_j-\alpha_j\alpha_i)=0$ and $(\beta_i\beta_j-\beta_j\beta_i)=0$. The last term is $(\alpha_i\beta_j-\alpha_j\beta_i)(p)=0$ because $\alpha_k(p)=0$ for every $k=1,...,n_1$.
			\end{itemize}
		\end{proof}

		\begin{theorem}
			\label{notas}
			Let $C$ be a trigonal curve of genus $g\geq 7$ as in \eqref{trigonal1} and let us consider the space $V=\langle \xi_p^1,\xi_p^2 \rangle \subset H^1(T_C)$. The only asymptotic direction in $V$ are in $\langle\xi_p^1\rangle$.
		\end{theorem}
		
		\begin{proof}
			Let $v=a\xi_p^1+b\xi_p^2\in V$, where $b\neq 0$. From Proposition \ref{vanishing} we immediately get $\rho(Q)(\xi_p^1\odot \xi_p^2)=0$ for every $Q\in I_2$. Hence $\rho(Q)(v\odot v)=b^2\rho(Q)(\xi_p^2\odot \xi_p^2)$. From Proposition \ref{remarkfrediani}, we have
			\begin{equation}
				\label{rhoq}
				\rho(Q)(\xi^2_p \odot \xi^2_p) =  2 \pi i  \left[ \sum_{k=0}^{1} \left (  \sum^u_{r,m=1}c_{rm} g^{(4-k)}_r(0) g_m^{(k)}(0)\right)\frac{(2-k)}{k!(4-k)!}\right]
			\end{equation}
			Let us focus on the term inside the round brackets. We see that we must compute $\sum\limits_{r,m=1}^u c_{rm} g^{(4)}_r(0) g_m(0)$ and $\sum\limits _{r,m=1}^u c_{rm} g^{'''}_r(0) g_m^{'}(0)$. In order to do that, we use again \eqref{eq2}, where now $h=4,l=0$ and $h=3,l=1$. 
			In both cases we can have $n+m=0,1,2,3,4$. Arguing exactly in the same way as in Proposition \ref{vanishing}, one can check that the only case that matters is $n+m=1$, the others are all zero. 
			\begin{itemize}
				\item Let us consider $\sum \limits_{r,m=1}^u c_{rm} g_r^{(4)} g_m(p)$. Here, since $n+m=1$, the only possibility is $n=1,m=0$. From Equation \eqref{eq2}, we have 
				\begin{multline}
					\sum \limits_{r,m=1}^u c_{rm} g_r^{(4)} g_m(p)=  \binom{4}{1}s^{(3)}(p)t(p)\cdot \\ [\sum \limits_{1\leq i<j\leq n_1} a_{ij}(\alpha_i^{'}\alpha_j-\alpha_j^{'}\alpha_i)(p)+ \\ +\sum \limits_{1\leq i<j\leq n_2} a'_{ij}(\beta_i^{'}\beta_j-\beta_j^{'}\beta_i)(p)+\sum \limits_{\substack{1 \leq i\leq n_1 \\   1\leq j \leq n_2}} a''_{ij}(\alpha_i\beta_j'-\alpha_j' \beta_i)(p)].
				\end{multline}
				Let us focus on the terms inside the square brackets. By Remark \ref{basis}, $ord_p(\alpha_i')=3i-3$ and $ord_p(\beta_i')=3i-4$. Hence if $i\geq 2$, $\alpha_i'(p)=\beta_i'(p)=0$, moreover if $i=1$, we have $j\geq 2$ and so $\alpha_j'(p)=\beta_j'(p)=0$. We get that the first two summands are zero. In the last one, the only term that may survive is $-4a''_{11}s^{(3)}(p)t(p)(\alpha_1'\beta_1)(p)$. Since $s^{(3)}(p)t(p)\alpha_1'\beta_1(p)\neq 0$, it depends whether $a''_{11}$ is zero or not.
				\item Now let us consider $\sum \limits_{r,m=1}^u c_{rm} g_r^{(3)} g_m
				'(p)$. Here, the possibilities are $n=1,m=0$ or $n=0,m=1$: but in the first case we get $(s''t'-t''s')(p)=0$ by the choice of $s$. By Equation \eqref{eq2}, we get
				\begin{multline}
					\sum \limits_{r,m=1}^u c_{rm} g_r^{(3)} g_m'(p)=  s^{(3)}(p)t(p)\cdot \\ [\sum \limits_{1\leq i<j\leq n_1} a_{ij}(\alpha_i^{'}\alpha_j-\alpha_j^{'}\alpha_i)(p)+ \\ +\sum \limits_{1\leq i<j\leq n_2} a'_{ij}(\beta_i^{'}\beta_j-\beta_j^{'}\beta_i)(p)+\sum \limits_{\substack{1 \leq i\leq n_1 \\   1\leq j \leq n_2}} a''_{ij}(\alpha_i^{'}\beta_j-\alpha_j \beta_i')(p)].
				\end{multline}
				As before, the only term that survives is 
				$-a''_{11}s^{(3)}(p)t(p)(\alpha_1'\beta_1)(p)$. 
			\end{itemize}
			Hence \eqref{rhoq} becomes 
			\begin{equation}
				\rho(Q)(\xi_p^2\odot \xi_p^2)=-\pi ia''_{11}s^{(3)}(p)t(p)(\alpha_1'\beta_1)(p). 
			\end{equation}
			Hence if we take $Q''_{11}$, we get $$\rho(Q''_{11})(\xi_p^2\odot \xi_p^2)=-\pi is^{(3)}(p)t(p)(\alpha_1'\beta_1)(p)\neq 0$$
			and $v$ is not asymptotic, as desired. Notice that $Q''_{11}$ really exists, since $g\geq 7$ and hence $n_1,n_2\geq 1$.
		\end{proof}
		Let us now consider every cyclic trigonal curve, with equation as in \eqref{cyctrigonal}. In this case, recall that by Remark \ref{trig2trig3} we have
			$H^0(M)=W_1\bigoplus W_2$, where $$W_1=\langle x\frac{dx}{y},x^2\frac{dx}{y},...,x^{n_1}\frac{dx}{y}\rangle$$ and 
			$$
			W_2=\langle x\prod_{i=r_1+1}^{r}(x-t_i)\frac{dx}{y^2},x^2\prod_{i=r_1+1}^{r}(x-t_i)\frac{dx}{y^2},...,x^{n_2}\prod_{i=r_1+1}^{r}(x-t_i)\frac{dx}{y^2}\rangle.
			$$
		\begin{remark}
			\label{rem1}
			Clearly, $ord_p(\alpha_i)=3i-2$ as in the previous case, since the holomorphic forms are the same. Notice that the same happens for $\prod_{i=r_1+1}^r(x-t_i)\beta_i$, because $\prod_{i=r_1+1}^r(x-t_i)$ does not vanish at $p$. Hence we also have $ord_p(\beta_i)=ord_p(\prod_{i=r_1+1}^r(x-t_i)\beta_i)=3i-3$.
		\end{remark}
		From Remark \ref{rem1}, we easily get the analogous results of Proposition \ref{vanishing} and Theorem \ref{notas} for every cyclic trigonal curve.
		\begin{proposition}
			Let $C$ be every cyclic trigonal curve, $p\in C$ a Weierstrass point and let $Q\in I_2$. Let $\gamma_1,...,\gamma_u \in H^0(K_C)$ where locally $\gamma_r=g_rdz$ such that $$Q=\sum \limits_{r,m=1}^u c_{rm} \gamma_r \otimes \gamma_m=\sum \limits_{1\leq i<j\leq n_1} a_{ij}Q_{ij}+\sum \limits_{1\leq i<j\leq n_2} a'_{ij}Q'_{ij}+\sum \limits_{\substack{1 \leq i\leq n_1  \\ 1\leq j \leq n_2}} a''_{ij}Q''_{ij}$$ Then $$
			\sum \limits_{r,m=1}^u c_{rm} g_r^{(h)} g_m^{(l)}(p)=0 \quad \forall h+l\leq 3.
			$$
		\end{proposition}
		Hence we have the following  theorem.
		\begin{theorem}
			\label{theochap3}
            Let $C$ be any trigonal curve of genus $g\geq 7$
			which is a cyclic cover of $\mathbb{P}^1$. For every $p\in C$ ramification point of $\pi$, the only asymptotic direction in $\langle \xi_p,\xi_p^2\rangle$ are in $\langle\xi_p\rangle$.
		\end{theorem}
        \begin{proof}
            By Remark \ref{rem1}, just follow the computations in Theorem \ref{notas}, since the only thing that we use is the vanishing order in $p$ of $\alpha_i,\beta_i$. 
        \end{proof}
		With similar computations, we are also able to prove that $\xi_p^3$ and $\xi_p^4$ are not asymptotic directions for each of the three kinds of trigonal curves. 
		\begin{theorem}
			\label{xi3xi4}
			Let $C$ be any cyclic trigonal curve of genus $g\geq 10$. Then $\xi_p^3$ and $\xi_p^4$ are not asymptotic directions.
		\end{theorem}
		\begin{proof}
			In order to show that $\xi_p^3$ and $\xi_p^4$ are not asymptotic, we need to find quadrics $Q_1,Q_2\in I_2$ such that $\rho(Q_1)(\xi_p^3\odot \xi_p^3)\neq 0$ and $\rho(Q_2)(\xi_p^4\odot \xi_p^4)\neq 0$. We prove it for curves as in \eqref{trigonal1}, the other cases being the same.
			\begin{itemize}
				\item First consider the case of $\xi_p^3$: let us take $Q'_{12}\in \Lambda^2W_2$. We notice that $Q'_{12}$ exists since $g\geq 10$, which implies $n_2\geq 2$. In order to use Proposition \ref{remarkfrediani}, we need to show that $\sum \limits_{r,m=1}^u c_{rm} g_r^{(h)} g_m^{(l)}(p)=0 \quad \forall h+l\leq 5.$ Since we already proved it for $h+l\leq 4$ in Theorem \ref{notas}, we are only left with $h+l=5$. From Equation \eqref{eq2}, we have 
				$$\sum \limits_{r,m=1}^u c_{rm} g_r^{(h)} g_m^{(l)}(p)=\sum\limits_{n=0}^h\sum\limits_{m=0}^l  \binom{h}{n} \binom{l}{m}(s^{(h-n)}t^{(l-m)}-t^{(h-n)}s^{(l-m)})(p)(a'_{12}(\beta_1^{(n)}\beta_2^{(m)}-\beta_2^{(n)}\beta_1^{(m)})(p).$$ Here the only non trivially zero case is $n+m=2$, hence $n=2,m=0$ or conversely. In both cases we get $\beta_2(p)=\beta_2^{(2)}(p)=0$, so the term is zero.\\ We are able to apply Proposition \ref{remarkfrediani} and we get
				$$\rho(Q'_{12})(\xi_p^3\odot \xi_p^3)=2 \pi i  \left( \sum_{k=0}^{2} \left (  \sum^u_{r, m=1}c_{rm} g^{(6-k)}_r(0) g_m^{(k)}(0)\right)\frac{(3-k)}{k!(6-k)!}\right).$$ For $k=0$, in the term inside the interior brackets, we have $$\sum \limits_{r,m=1}^u c_{rm} g_r^{(6)} g_m(p)=\sum\limits_{n=0}^h\sum\limits_{m=0}^l  \binom{h}{n} \binom{l}{m}(s^{(h-n)}t^{(l-m)}-t^{(h-n)}s^{(l-m)})(p)(\beta_1^{(n)}\beta_2^{(m)}-\beta_2^{(n)}\beta_1^{(m)})(p).$$ The only non trivial term is for $n=3,m=0$, which yields $-\binom{6}{3}s^3(p)t(p)\beta_1(p)\beta_2^{(3)}(p)\neq 0$. While if $k=1,2$ it is easy to see that we get zero using the vanishing of $\beta_i$ and its derivatives. We finally have 
				$$\rho(Q'_{12})(\xi_p^3\odot \xi_p^3)=-\pi i \binom{6}{3}s^3(p)t(p)\beta_1(p)\beta_2^{(3)}(p)\neq 0,$$ as desired.
				\item Now consider the case of $\xi_p^4$. Let us now take $Q_{12}\in \Lambda^2W_1$: indeed, $Q_{12}$ exists since $g\geq 10$ and hence $n_1\geq 2$. In order to compute $\rho(Q_{12})(\xi_p^4\odot \xi_p^4)$, by Proposition \ref{remarkfrediani} one need to show that
				$\sum \limits_{r,m=1}^u c_{rm} g_r^{(h)} g_m^{(l)}(p)=0 \quad \forall h+l\leq 7.$ This is easy to prove, using the vanishing of $\alpha_i$ and its derivatives. Finally, one gets
				$$\rho(Q_{12})(\xi_p^4\odot \xi_p^4)=-\pi i \binom{7}{4}s^{3}(p)t(p)\alpha_1'(p)\alpha_2^{(4)}(p)\neq 0.$$
				We omit the details, since the computations are very similar to the previous case.
			\end{itemize}
		\end{proof}
		As a corollary of Theorem \ref{notas} and Theorem \ref{xi3xi4} we have this result.
		\begin{theorem}
			\label{theoxi1x4}
			Let $C$ be any trigonal curve of genus $g\geq 10$ which is a cyclic cover of $\mathbb{P}^1$. For every ramification point $p$ of $\pi$, the only asymptotic directions in $\langle \xi_p^1,\xi_p^2,\xi_p^3,\xi_p^4\rangle$ are in $\langle\xi_p\rangle$.
		\end{theorem}
		\begin{proof}
			First, let us consider the subspace $\langle \xi_p^1,\xi_p^2,\xi_p^3\rangle$ and take $v=a_1\xi_p^1+a_2\xi_p^2+a_3\xi_p^3$ an asymptotic direction which is not in $\langle \xi_p\rangle$. Then we must have $a_3\neq 0$, since otherwise $v\in \langle \xi_p\rangle$ by Theorem \ref{notas}. We have
			\begin{multline*}
				\rho(Q)(v\odot v)=a_1^2\rho(Q)(\xi_p^1\odot \xi_p^1)+2a_1a_2\rho(Q)(\xi_p^1\odot \xi_p^2)+2a_1a_3\rho(Q)(\xi_p^1\odot \xi_p^3)+\\+2a_2a_3\rho(Q)(\xi_p^2\odot \xi_p^3)+a_2^2\rho(Q)(\xi_p^2\odot \xi_p^2)+a_3^2\rho(Q)(\xi_p^3\odot \xi_p^3).
			\end{multline*}
			If we take $Q=Q'_{12}$, from Theorem \ref{xi3xi4}, we get that all the terms are zero except for $a_3^2\rho(Q'_{12})(\xi_p^3\odot \xi_p^3)$ which is different from zero. Since $a_3\neq0$, we are done.\\
			Now consider $\langle \xi_p^1,\xi_p^2,\xi_p^3,\xi_p^4\rangle$ and take $v=a_1\xi_p^1+a_2\xi_p^2+a_3\xi_p^3+a_4\xi_p^4$. Here we can assume $a_4\neq 0$ by the previous case, then the proof follows as before by taking $Q_{12}$ and applying Theorem \ref{xi3xi4}.
		\end{proof}
		As a final result, we will show that for a trigonal curve as in \eqref{trigonal1}, the space $\langle \xi_p^1,\xi_p^2,\xi_p^3,\xi_p^4\rangle$ is isotropic for $Ker(\mu_2)$.
		\begin{remark}
			\label{coeff2}
			Assume $Q\in Ker(\mu_2)$, this implies $\psi^{-1}(Q)\in Ker(\mu_{1,M})$. Then if $\psi^{-1}(Q)\in \Lambda^2 W_1$ we have $a_{12}=a_{13}=0$. If $\psi^{-1}(Q)\in \Lambda^2 W_2$ we have $a'_{12}=a'_{13}=0$. While if $\psi^{-1}(Q)\in W_1\otimes W_2 $ we get $a''_{11}=a''_{12}=a''_{21}=0$. This follows immediately by Theorem \ref{mu2tensor} and Proposition \ref{eqmu2k-1}.
		\end{remark}
		\begin{proposition}
			With the same notation and assumptions of Proposition \ref{vanishing}. If $Q\in Ker(\mu_2)$, then $$\rho(Q)(\xi_p^h\odot \xi_p^l)=0\quad \forall h+l\leq 9.$$
		\end{proposition}
		\begin{proof}
			By Proposition \ref{remarkfrediani}, it suffices to show that 
			$$
			\sum \limits_{r,m=1}^u c_{rm} g_r^{(h)} g_m^{(l)}(p)=0 \quad \forall h+l\leq 9.
			$$
			For reader's convenience, we recall Equation \eqref{eq2}.
			\begin{multline}
				\sum \limits_{r,m=1}^u c_{rm} g_r^{(h)} g_m^{(l)}(p)=\sum\limits_{n=0}^h\sum\limits_{m=0}^l  \binom{h}{n} \binom{l}{m}(s^{(h-n)}t^{(l-m)}-t^{(h-n)}s^{(l-m)})(p)\cdot \\ [\sum \limits_{1\leq i<j\leq n_1} a_{ij}(\alpha_i^{(n)}\alpha_j^{(m)}-\alpha_j^{(n)}\alpha_i^{(m)})(p)+ \\ +\sum \limits_{1\leq i<j\leq n_2} a'_{ij}(\beta_i^{(n)}\beta_j^{(m)}-\beta_j^{(n)}\beta_i^{(m)})(p)+\sum \limits_{\substack{1 \leq i\leq n_1 \\   1\leq j \leq n_2}} a''_{ij}(\alpha_i^{(m)}\beta_j^{(n)}-\alpha_j^{(n)}\beta_i^{(m)})(p)].
			\end{multline}
			where $h+l\leq 9$. So we need to show that the right hand side is zero. Arguing in a very similar way as in Proposition \ref{vanishing}, we see that we only have to study the cases when $n+m\leq 6$, since the others are trivially zero. We will focus separately on the three summands in the square brackets.
			\begin{itemize}
				\item Consider $$\sum \limits_{1\leq i<j\leq n_1} a_{ij}(\alpha_i^{(n)}\alpha_j^{(m)}-\alpha_j^{(n)}\alpha_i^{(m)})(p).$$ We may assume $m<n$, since if $m=n$ the expression is clearly zero; so $m<n\leq 6$. Assume first $i=1$. Then by Remark \ref{coeff2}, we must have $j\geq 4$. By Remark \ref{basis}, $\alpha_j^{(n)}(p)=0$ if $n\leq 3j-3$. If $j\geq 4$, $3j-3>6$, hence $\alpha_j^{(m)}(p)=\alpha_j^{(n)}(p)=0$ and we are done. If $i\geq 2$, then $j\geq 3$ and by the same argument we get that everything is zero, because $3j-3\geq 6$. So the first summand is always zero.
				\item Now consider $$\sum \limits_{1\leq i<j\leq n_2} a'_{ij}(\beta_i^{(n)}\beta_j^{(m)}-\beta_j^{(n)}\beta_i^{(m)})(p).$$ Again we may assume $m<n\leq 6$. As before, if $i=1$ we have $j\geq 4$ and $\beta_j^{(n)}(p)=0$ if $n\leq 3j-4$. Since $3j-4\geq 8$, we have zero. If we assume $i\geq 2$, then $j\geq 3$ and $\beta_j^{(m)}(p)=\beta_j^{(n)}(p)=0$ if $n\leq 5$. We are only left with the case $n=6$ and $m=0$: but this is zero because $\beta_i(p)=\beta_j(p)=0$ if $i\geq 2$.
				\item Finally, we have $$\sum\limits_{\substack{1 \leq i\leq n_1 \\   1\leq j \leq n_2}} a''_{ij}(\alpha_i^{(m)}\beta_j^{(n)}-\alpha_j^{(n)}\beta_i^{(m)})(p).$$ First, notice that here we cannot assume $m<n$. If $i=1$, now $j\geq 3$ by Remark \ref{coeff2}. We have $\alpha_j^{(n)}(p)=0$ if $n\leq 3j-3$ and $\beta_j^{(n)}(p)=0$ if $n\leq 3j-4$. Since $3j-3\geq 6$ and $3j-4\geq 5$, the only case that matters is $n=6,m=0$. This is clearly zero, since $\alpha_1(p)=0$. Actually, since if $i\geq 2$, $\alpha_i(p)=0$, the same argument shows that the expression is zero if $j\geq 3$. By symmetry, the same is true if $i\geq 3$. Hence the last case remaining is $i=j=2$. Now we can assume $m<n$: hence we must have $m\leq 2$, since $n+m\leq 6$. Since $a_2^{(m)}(p)=\beta_2^{(m)}(p)=0$ if $m\leq 2$, we are done. 
			\end{itemize}
		\end{proof}
		\begin{corollary}
			Let $C$ be a trigonal curve as in \eqref{trigonal1}. Then the space $\langle\xi_p^1,\xi_p^2,\xi_p^3,\xi_p^4\rangle$ is isotropic for $\rho(Q)$, for all $Q\in Ker(\mu_2)$.
		\end{corollary}
		We conclude with a final remark and a conjecture.
		\begin{remark}
			Families of cyclic trigonal curves for $g\geq 7$ are not totally geodesic. This was already known by (\cite{doi:10.1142/S0129167X15500056}, Corollary 5.9), but we obtained a new proof. They have dimension $r-3=g-1$.\\ If we denote by $\mathcal{X}$ such a family, we have $T_{[C]}\mathcal{X}=(H^1(T_C))^{\mathbb{Z}/3\mathbb{Z}}$, where we mean the invariant part with respect to the action of $\mathbb{Z}/3\mathbb{Z}$. We claim that $\xi_p^2\in (H^1(T_C))^{\mathbb{Z}/3\mathbb{Z}}$. Indeed, if we denote by $\zeta$ a generator of $\mathbb{Z}/3\mathbb{Z}$, defining $w=\zeta z$ we have $$
			\zeta\cdot \xi_p^2=\frac{\overline{\partial}\rho}{w^2}\frac{\partial}{\partial w}=\frac{\overline{\partial}\rho}{\zeta^2z^2}\frac{1}{\zeta}\frac{\partial}{\partial z}=\frac{\overline{\partial}\rho}{z^2}\frac{\partial}{\partial z}=\xi_p^2.
			$$
			Hence we get $\langle \xi_p^1, \xi_p^2,\xi_p^3,\xi_p^4\rangle \cap (H^1(T_C))^{\mathbb{Z}/3\mathbb{Z}}=\langle \xi_p^2\rangle $.\\
			From Theorem \ref{notas}, we have $Q''_{11}\in (I_2)^{\mathbb{Z}/3\mathbb{Z}}$ such that $\rho(Q''_{11})(\xi_p^2\odot \xi_p^2)\neq 0$. 
			This implies that $\xi_p^2$ is a tangent direction which is not annihilated by the second fundamental form, so $\mathcal{X}$ is not totally geodesic. 
		\end{remark}
		We conjecture that for every cyclic trigonal curve the only asymptotic direction in $\langle \xi_p^1,...,\xi_p^l\rangle$ is $\xi_p^1\quad \forall l<2g-2$. We have proven that we have the following chain of inclusions:
		$$
		Ker(\mu_{2\lfloor\frac{n_2}{2}\rfloor})\subsetneq Ker(\mu_{2\lfloor\frac{n_2}{2}\rfloor-2})\subsetneq...\subsetneq I_2.
		$$
	Having a complete description of $Ker(\mu_{2k})$ as in the hyperelliptic case (see \cite{faro2025highergaussianmapshyperelliptic}) would be very helpful to go on with similar computations and hopefully prove the claim. Unfortunately, computing the equations of $\mu_{2k}(Q)$ is quite difficult for $Q\in \psi (W_1\otimes W_2)$.

    \bibliographystyle{alpha} 
	\bibliography{references} 
\end{document}